\documentclass[12pt,reqno,a4paper]{amsart}
\usepackage{amsmath,amssymb,amsthm,amsaddr}
\usepackage{enumitem}
\usepackage[margin=2.7cm,top=3.5cm,bottom=3.5cm,footskip=1cm,headsep=1.5cm]{geometry}
\usepackage[utf8]{inputenc}

\usepackage{amsthm}
\usepackage[british]{babel}

\usepackage{xcolor}

\usepackage{tikz}
\usepackage{comment} 
\usepackage{setspace} 

\author{H. Egger$^{1,2}$, S. Kurz$^3$, R. L\"oscher$^4$}
\address{%
$^1$Institute of Numerical Mathematics, Johannes Kepler University Linz, Austria \\
$^2$Johann Radon Institute for Computational and Applied Mathematics, Linz, Austria\\
$^3$Faculty of Information Technology, University of Jyväskylä, Finland\\
$^4$Institute of Numerical Mathematics, TU Graz, Austria}
\email{herbert.egger@jku.at}
\email{stefan.m.kurz@jyu.fi}
\email{loescher@math.tugraz.at}

\title[Exponential stability of damped wave equations]{On the exponential stability \\of uniformly damped wave equations}

\newtheorem{lemma}{Lemma}[section]
\newtheorem{theorem}[lemma]{Theorem}

\theoremstyle{definition}
\newtheorem{remark}[lemma]{Remark}

\def\dt{\partial_t}
\def\dtt{\partial_{tt}}

\def\ddt{\frac{\mathrm{d}}{\mathrm{d}t}}
\def\dtau{d_\tau}
\def\ddtau{d_{\tau\tau}}
\def\div{\operatorname{div}}
\def\curl{\mathrm{curl}}
\def\eps{\epsilon}

\def\d{\operatorname{d}}
\def\cb{c_{\beta}}
\def\Cb{C_{\beta}}
\def\CP{C_{\mathrm{P}}}

\def\E{\mathcal{E}}
\def\N{\mathcal{N}}
\def\P{\mathcal{P}}
\def\T{\mathcal{T}}
\def\R{\mathcal{R}}
\def\N{\mathcal{N}}
\def\D{\mathcal{D}}
\def\RT{\mathcal{RT}}
\def\AAW{\mathcal{AAW}}
\def\RR{\mathbb{R}}

\begin{document}

\begin{abstract}
We study damped wave propagation problems phrased as abstract evolution equations in Hilbert spaces. Under some general assumptions, including a natural compatibility condition for initial values, we establish exponential decay estimates for all mild solutions using the language and tools of Hilbert complexes. This framework turns out strong enough to conduct our analysis but also general enough to include a number of interesting examples. Some of these are briefly discussed. By a slight modification of the main arguments, we also obtain corresponding decay results for numerical approximations obtained by compatible discretization strategies. 
\end{abstract}

\maketitle

\section{Introduction}
\label{sec:intro}

Wave propagation through 
media is typically accompanied by some sort of damping, e.g., through friction, conduction, etc. This leads to dissipation of energy and eventually to convergence of the system to a steady state. 
In this paper, we study a general class of damped wave propagation problems of the common abstract form 
\begin{alignat}{2}
\alpha \dt u - \d^* u^* &= -\beta u  \qquad && \text{in } W^1, \ t \ge 0, \label{eq:1}\\
\gamma \dt u^* + \d u &= 0 \qquad && \text{in } W^2, \ t \ge 0. \label{eq:2}
\end{alignat}
Here $W^1$, $W^2$ are two Hilbert spaces, 
$\d : W^1 \to W^2$ 
is a densely defined and closed linear operator, and 
$\d^* : W^2 \to W^1$ its adjoint. 
Furthermore, $\alpha,\beta: W^1 \to W^1$ and $\gamma: W^2 \to W^2$ are selfadjoint positive isomorphisms, generating scalar products and norms on $W^1$ and $W^2$, respectively. 
The spaces and operators form a segment of a Hilbert complex \cite{Arnold2010,Bruening1992,Holst2012} 
\begin{figure}[ht!]
\centering
\begin{tikzpicture}
\node (W1) at (0,2) {$W^1$};
\node (W2) at (2,2) {$W^2$};
\node (W1s) at (0,0) {$W^1$}; 
\node (W2s) at (2,0) {$W^2$}; 
\draw [thick,->] (W1) -- node[above,midway] {$\d$} (W2);
\draw [thick,->] (W2s) -- node[above,midway] {$\d^*$} (W1s);
\draw [thick,->] (W1) -- node[left,midway] {$\alpha$} node[right,midway] {$\beta$}(W1s);
\draw [thick,->] (W2s) -- node[right,midway] {$\gamma$} (W2);
\end{tikzpicture}
\end{figure}
and the operators $\alpha, \beta$ and $\gamma$ map isomorphically between the primal and the dual complex, i.e, between the upper and lower row of the illustration. 
This setting turns out strong enough to analyse the long time stability of \eqref{eq:1}--\eqref{eq:2}, but at the same time, general enough to cover a variety of interesting examples, ranging from electromagnetics, to acoustics and elastodynamics, as well as their discretization, and even to certain differential equations on graphs. 

As a prototypical example for our setting, let us consider Maxwell's equations in a linear conducting medium. The governing equations read~\cite{Dautray5,Stratton}
\begin{alignat}{2}
\eps \dt E - \curl H &= -\sigma E  \qquad &&\text{in } \Omega, \ t \ge 0, \label{eq:3} \\
\mu \dt H + \curl_0 E &= 0 \qquad  && \text{in } \Omega, \ t \ge 0. \label{eq:4}
\end{alignat}
Here $\curl_0 E$ denotes the $\curl$ operator with zero boundary conditions $E \times n=0$.
Sufficiently smooth solutions of \eqref{eq:3}--\eqref{eq:4} can be shown to satisfy
\begin{align} \label{eq:5}
\ddt \frac{1}{2}(\|E(t)\|_\eps^2 + \|H(t)\|_\mu^2) = - \|E(t)\|_\sigma^2, \qquad t \ge 0,
\end{align}
where $\|u\|_\kappa^2 = \int_\Omega \kappa |u|^2 \, dx$ denotes the norm generated by a coefficient $\kappa$. 
The non-increase of the energy is directly encoded in the port-Hamiltonian structure of the system \cite{jacob2012linear,Macchelli2004,Rashad2020}, which becomes evident in the weak formulation of the problem.

A similar power balance also holds for systems of the abstract form \eqref{eq:1}--\eqref{eq:2}. Under the general assumptions mentioned above and a natural compatibility condition on the initial data, we can even establish exponential decay of the energy. 
\begin{theorem} \label{thm:main} 
Any mild solution $(u,u^*) \in C([0,\infty);W^1 \times W^2)$ of \eqref{eq:1}--\eqref{eq:2}, which satisfies the compatibility condition $\gamma u^*(0) \in \R(\d)$, decays exponentially, i.e., 
\begin{align} \label{eq:decay}
\|u(t)\|_\alpha^2 + \|u^*(t)\|_{\gamma}^2 \le C' e^{-c'(t-s)} \big(\|u(s)\|_\alpha^2 + \|u^*(s)\|_{\gamma}^2\big), 
\end{align}
for all $0 \le s \le t$ with constants $C'$, $c' >0$ independent of the particular solution. 
\end{theorem}

The setting of Hilbert complexes has been used very successfully to study a variety of systems of partial differential equations arising in the natural sciences, both, from an analytical and a numerical point of view; see~\cite{Arnold2018,Arnold2010,Bruening1992} for details and references. 
For instance, it allows to establish existence of mild and classical solutions of \eqref{eq:1}--\eqref{eq:2} by means of semigroup theory~\cite{Dautray5,Pazy83}. 
In this paper, we aim to utilize the framework of Hilbert complexes to prove exponential decay of uniformly damped wave equations in an abstract setting. On the one hand, this approach allows to clarify the main ingredients required for the analysis, and on the other hand, it is transferable to discretizations. 

\subsection*{Main arguments}
For the proof of our main result, we use refined energy estimates which exploit the transfer between the two components of the energy through the differential operators $\d$ and $\d^*$. We further employ a variational characterization of solutions and a small number of assumptions, which can be verified for many applications and for appropriate discretizations. 
For some particular examples, a proof of \eqref{eq:decay} can already be found in the literature: 
In \cite{Eller2019}, the decay estimate of our main theorem was shown for the Maxwell system but via different arguments. 
In that case, the compatibility condition amounts to $B(0) = \mu H(0) = \curl_0 A(0)$ for some vector potential $A(0)$ which, in particular, implies $\div B(0)=0$, and this condition  has a clear physical meaning.
Similar decay estimates can also be found for other types of systems; see, e.g., \cite{Kugler2018,Haraux1987,Zuazua1990}. 
Results for system with strong damping or  boundary damping can be found, for instance, in \cite{Ervedoza2009,Ikehata2013,Lagnese83}.
An exponential stability result for abstract evolution problems was established in \cite[Thm. 2.3]{Humaloja2019}.
The main theorem of the paper generalizes some of these results and simplifies the proofs, providing more insight into the underlying mathematical structures and tighter estimates for the decay rate. 

\subsection*{Outline}
The remainder of the manuscript is organized as follows: 
In Section~\ref{sec:prelim},
we formally introduce our assumptions and some preliminary results required for the proof of Theorem~\ref{thm:main},
which is presented in Section~\ref{sec:proof}.
In Section~\ref{sec:disc}, we show that the exponential stability result and its proof carry over almost verbatim to appropriate discretizations. 
A couple of further examples is presented in Section~\ref{sec:examples}, for which our theoretical results apply immediately, 
and we close the presentation with a short summary and outlook to further possible extensions.

\section{Preliminaries and notation}
\label{sec:prelim}

Let us briefly introduce our main assumptions and some preliminary results required for the proof of our main theorem. Further details can be found in \cite{Arnold2018,Arnold2010}, for instance.

\subsection{Assumptions}\label{subsec:assumptions}
We consider real Hilbert spaces $W^1$, $W^2$ with scalar products denoted by $\langle w,w^*\rangle_{W^k}$, $k=1,2$. 
Further let $\mathrm{d}:V^1 \subset W^1 \to W^2$ be a densely defined closed linear operator with domain $V^1=\{v^1 \in W^1: d v^1 \in W^2\}$. By 
\begin{align} \label{eq:adjoint}
    \langle \d^* u^*, u\rangle_{W^1} = \langle u^*, \d u\rangle_{W^2} \qquad \forall u \in V^1, \ u^* \in V_2^*
\end{align}
we define the action of the adjoint operator $\d^*: V_2^* \subset W^2 \to W^1$, which is again densely defined and closed~\cite{Yosida}.  
Finally, let $\alpha,\beta : W^1 \to W^1$ and $\gamma : W^2 \to W^2$ be selfadjoint and positive isomorphisms. We denote by
\begin{align} \label{eq:norms}
\langle u,v\rangle_\alpha & =\langle \alpha u, v \rangle_{W^1}, \qquad \|u\|_\alpha^2 = \langle u,u\rangle_\alpha
\end{align}
the scalar product and norm generated by the operator $\alpha$. By the previous assumptions, they are equivalent to the natural scalar product and norm of $W^1$. 
In a similar manner, the operators $\beta$ and $\gamma$, and the inverses $\alpha^{-1}$, $\gamma^{-1}$ define equivalent scalar products and norms on $W^1$ and $W^2$, respectively.

\subsection{Preliminaries}
Under the conditions above, which we assume to hold throughout the manuscript, the norms generated by $\alpha$ and $\beta$ are equivalent. In particular
\begin{align}\label{eq:norm_equivalence}
\cb \|u\|_\alpha^2 \le \|u\|_\beta^2
\le \Cb \|u\|_\alpha^2 
\qquad \forall u \in W^1
\end{align}
with equivalence constants
$0<\cb\leq C_\beta$, which are introduced here for later reference. 
Our general assumptions further imply the validity of a Poincar\'e inequality
\begin{align} \label{eq:poincare}
\|u\|_\alpha \le \CP \|\d u\|_{\gamma^{-1}} \qquad \forall u \in N(\d)^{\perp_\beta},
\end{align}
where $N(\d)^{\perp_\beta}=\{v \in V^1 : \langle \beta v, z\rangle_{W^1}=0 \ \forall z \in N(\d)\}$; see \cite[Ch.~4]{Arnold2018} for details.

\begin{remark}
For the Maxwell system mentioned in the introduction, the constants that appear in the assumptions can be given a physical interpretation, as follows:
\begin{equation}\label{eq:time_constants}
\cb=\frac{1}{\tau_\mathrm{r,max}}\,,\quad
\Cb=\frac{1}{\tau_\mathrm{r,min}}\,,\quad
\CP=\Delta t\,,
\end{equation}
where $\tau_\mathrm{r,max}$, $\tau_\mathrm{r,min}$ are the maximal and minimal dielectric relaxation times $\sim\eps/\sigma$, respectively, and $\Delta t$ is a characteristic traversal time of light 
through the domain.
\end{remark}

As a first step of our analysis, we discuss the existence of solutions to problem \eqref{eq:1}--\eqref{eq:2}, which can be proven using basic results of semigroup theory; see, e.g., \cite{Dautray5,Pazy83}. 
\begin{lemma}
For any $(u_0,u_0^*) \in W^1 \times W^2$, the system \eqref{eq:1}--\eqref{eq:2} has a unique mild solution $(u,u^*) \in C([0,\infty);W^1 \times W^2)$ with $u(0)=u_0$ and $u^*(0)=u_0^*$. 
If $(u_0,u_0^*) \in V^1 \times V_2^*$, then $(u,u^*) \in C^1([0,\infty);W^1 \times W^2) \cap C([0,\infty);V^1 \times V_2^*)$ is a classical solution.     
\end{lemma}
\begin{proof}
A similar proof can be found in \cite{Eller2019} for the special case of Maxwell's equations. We set $X=W^1 \times W^2$ and rewrite \eqref{eq:1}--\eqref{eq:2} compactly as 
\begin{align} \label{eq:evo}
\dt x &= A x \qquad \text{in } X, \ t \ge 0,
\end{align}
with operator $A : X \to X$ mapping $(u,u^*) \mapsto (\alpha^{-1} (\d^* u^* - \beta u), -\gamma^{-1} \d u)$. By the assumptions made in section \ref{subsec:assumptions}, the operator $A$ is closed and densely defined, with domain $\D(A)=V^1 \times V_2^*$. 
We equip the product space $X$ with the energy scalar product $\langle (u,u^*), (v,v^*)\rangle_X = \langle u,v\rangle_\alpha + \langle u^*,v^*\rangle_{\gamma}$.
Then 
\begin{align*}
\langle A (u,u^*), (u,u^*)\rangle_X 
&= -\|u\|_\beta^2 \le 0 \qquad \forall (u,u^*) \in \mathcal{D}(A),
\end{align*}
which shows that $A$ is dissipative. The corresponding adjoint operator $A^* : X \to X$ maps $(v,v^*) \mapsto (-\alpha^{-1}(\d^* v^* - \beta v), \gamma^{-1} \d v)$ and is dissipative as well. Hence by a corollary of the Lumer-Phillips theorem, $A$ generates a contraction semigroup on $X$ and the statements of the lemma follow immediately; see, e.g., \cite[Cor. 3.17]{Engel00}.
\end{proof}

\begin{remark} \label{rem:dense}
Let $A: \D(A) \subset X \to X$ be given as in the previous proof. Then any mild solution $u \in C([0,\infty);X)$ of \eqref{eq:evo} can be approximated in the norm of $C([0,\infty);X)$ by classical solutions $\tilde x \in C^1([0,\infty);X) \cap C([0,\infty);\D(A))$. This is a direct consequence of the density of $\D(A) \subset X$, which follows from that of $V^1 \subset W^1$ and $V_2^* \subset W^2$.
\end{remark}

The following weak characterization of classical solutions will be used for our analysis, but later also serves as the starting point for the design of discretization methods. 
\begin{lemma}
Let $(u,u^*)$ 
denote a classical solution of \eqref{eq:1}--\eqref{eq:2}. Then 
\begin{alignat}{2}
\langle\alpha\dt u(t), v\rangle_{W^1} - \langle u^*(t), \d v\rangle_{W^2} &= - \langle\beta u(t), v\rangle_{W^1}, \label{eq:var1} \\
\langle\gamma \dt u^*(t), v^*\rangle_{W^2} + \langle \d u(t), v^*\rangle_{W^2} &= 0, \label{eq:var2}
\end{alignat}
for all test functions $v \in V^1$, $v^* \in W^2$ and all $t \ge 0$.
Moreover, 
\begin{align} \label{eq:diss}
\ddt \frac{1}{2} \Big( \|u(t)\|_\alpha^2 &+ \|u^*(t)\|_\gamma^2 \Big) 
= -\|u(t)\|_\beta^2 \le 0. 
\end{align}
\end{lemma}
\begin{proof}
The variational identities follow immediately from testing the equations with $v$ and $v^*$, respectively, and using $\langle \d^* u^*(t), v\rangle_{W^1} = \langle u^*(t), \d v \rangle_{W^2}$, which follows from the definition of the adjoint operator. By formal differentiation, we obtain
\begin{align*}
\ddt \frac{1}{2} \Big( \|u(t)\|_\alpha^2 &+ \|u^*(t)\|_\gamma^2 \Big)  
 = \langle\alpha\dt u(t), u(t)\rangle_{W^1} + \langle\gamma\dt u^*(t), u^*(t)\rangle_{W^2} \\
&= -\langle\beta u(t),u(t)\rangle_{W^1} + \langle u^*(t), \d u(t)\rangle_{W^2} 
- \langle \d u(t), u^*(t)\rangle_{W^2},
\end{align*}
where we employed \eqref{eq:var1}--\eqref{eq:var2} with $v=u(t)$ and $v^* = u^*(t)$, respectively. The last two terms cancel each other, which already yields the power balance \eqref{eq:diss}. 
\end{proof}

By integration of \eqref{eq:diss} in time, we see that $\|u(t)\|_\alpha^2 + \|u^*(t)\|_\gamma^2 \le \|u(s)\|_\alpha^2 + \|u^*(s)\|_\gamma^2$ for all $0 \le s \le t$, i.e., the energy of the system is non-increasing, and by Remark~\ref{rem:dense}, this estimate carries over to mild solutions. 
Further assumptions and arguments are needed, however, to prove the exponential decay of the energy.

\section{Proof of the main result}
\label{sec:proof}

In this section, we will establish the exponential decay estimate \eqref{eq:decay} for an arbitrary classical solution $(u,u^*)$ of \eqref{eq:1}--\eqref{eq:2} satisfying the compatibility condition $\gamma u^*(0) \in \R(\d)$. The assertion of Theorem~\ref{thm:main} then follows by Remark~\ref{rem:dense}. 

\subsection{Auxiliary functions}
To simplify the energy estimates derived in the following, we 
define a primitive $(w,w^*)$ of the solution $(u,u^*)$ by integration in time, i.e. 
\begin{align} \label{eq:aux}
w(t) = w_0 + \int_0^t u(s) \, ds, \qquad w^*(t) = w^*_0 + \int_0^t u^*(s) \, ds.
\end{align}
The initial values $w_0$ and $w^*_0$ are chosen as a solution of 
\begin{align}
\beta w_0 - \d^* w^*_0 &= -\alpha u(0) \label{eq:aux01}\\
\d w_0 &= -\gamma u^*(0).  \label{eq:aux02}
\end{align}
By elementary arguments, one can verify the following assertions.
\begin{lemma} \label{lem:aux}
Let $(u,u^*)$ be a classical solution of \eqref{eq:1}--\eqref{eq:2} with $\gamma u^*(0) \in \R(\d)$. Then the system \eqref{eq:aux01}--\eqref{eq:aux02} has 
a unique solution $(w_0,w^*_0) \in V^1 \times (V_2^* \cap N(\d^*)^\perp)$.
Moreover, the function $(w,w^*)$ in \eqref{eq:aux} lies in $C^2([0,\infty);W^1 \times W^2) \cap C^1([0,\infty);V^1 \times V_2^*)$ and
\begin{alignat}{2}
\alpha \dt w - \d^* w^* &= -\beta w \qquad &&\text{in } W^1, \ t \ge 0, \label{eq:aux1}\\
\gamma \dt w^* + \d w &= 0 \qquad &&\text{in } W^2, \ t \ge 0. \label{eq:aux2}
\end{alignat}
Furthermore 
\begin{alignat}{2}
\alpha \dtt w - \d^* \dt w^* &= -\beta \dt w \qquad &&\text{in } W^1, \ t \ge 0, \label{eq:aux11}\\
\gamma \dtt w^* + \d \dt w &= 0 \qquad &&\text{in } W^2, \ t \ge 0. \label{eq:aux12}
\end{alignat}
\end{lemma}
\begin{proof}
Solvability of \eqref{eq:aux01}--\eqref{eq:aux02} follows from a generalization of Brezzi's lemma; see \cite[Thm.~4.2.4]{Boffi}. 
%
%
The regularity of $(w,w^*)$, on the other hand, follows immediately from that of $(u,u^*)$ and the definition of $(w,w^*)$.
Using \eqref{eq:aux} and \eqref{eq:1}, we further see that 
\begin{align*}
\alpha \dt w(t) 
&= \alpha u(t)  
 = \alpha \big(u(0) + \int_0^t \dt u(s) \, ds\big)  
 = \int_0^t \bigl(\d^* u^*(s) - \beta u(s)\bigr) \, ds + \alpha u(0)\\
&= \d^* w^* (t) - \beta w(t) + \big[\alpha u(0) - \d^* w^*(0) + \beta w(0)\big].  
\end{align*}
By equation~\eqref{eq:aux01}, the term in brackets vanishes, and we obtain  \eqref{eq:aux1}. The second equation \eqref{eq:aux2} follows similarly, and the remaining identities follow from differentiation of the previous ones. Due to the regularity of the functions, all steps are justified.
\end{proof}

\begin{remark}
Note that $u=\dt w$ and $u^*=\dt w^*$ by definition of $(w,w^*)$ in \eqref{eq:aux}. Hence the equations \eqref{eq:aux11}--\eqref{eq:aux12} are in fact equivalent to the original system \eqref{eq:1}--\eqref{eq:2}. 
\end{remark}

\subsection{Intermediate results}
With similar arguments as before, one can see that the functions $(w,w^*)$ defined in \eqref{eq:aux} satisfy the variational identities
\begin{alignat}{2}
\langle\alpha\dt w(t), v\rangle_{W^1} - \langle w^*(t), \d v\rangle_{W^2} &= - \langle\beta w(t), v\rangle_{W^1} \label{eq:var1w} \\
\langle\gamma\dt w^*(t), v^*\rangle_{W^2} + \langle \d w(t), v^*\rangle_{W^2} &= 0, \label{eq:var2w}
\end{alignat}
for all $v \in V^1$, $v^* \in V^*_2$, and all $t \ge 0$. Moreover, 
\begin{alignat}{2}
\langle\alpha\dtt w(t), v\rangle_{W^1} - \langle \dt w^*(t), \d v\rangle_{W^2} &= - \langle\beta \dt w(t), v\rangle_{W^1} \label{eq:var1w1} \\
\langle\gamma\dtt w^*(t), v^*\rangle_{W^2} + \langle \d \dt w(t), v^*\rangle_{W^2} &= 0, \label{eq:var2w1}
\end{alignat}
for all $v \in V^1$, $v^* \in V^*_2$, and all $t \ge 0$. As a direct consequence of the latter, we obtain
\begin{equation}\label{eq:dissw}
\ddt \frac{1}{2} \left( \|\dt w(t)\|_\alpha^2 + \|\dt w^*(t)\|_\gamma^2 \right)  
= -\|\dt w(t)\|_\beta^2 \le 0,
\end{equation}
which in fact is equivalent to the power balance~\eqref{eq:diss}.  
As noted before, this type of energy estimate is not sufficient, however, to prove exponential decay of the system. 

\subsection{Improved energy estimate}
For our analysis, we will use the modified energy
\begin{align} \label{eq:mod}
    \E_\delta(t) := \frac{1}{2} \left( \|\dt w(t)\|^2_\alpha + \|\dt w^*(t)\|_\gamma^2\right) + \delta \langle  \dt w(t),w(t)\rangle_\alpha,
\end{align}
where $\delta>0$ is a parameter to be chosen later on. The extra term provides a coupling between \eqref{eq:var1w}--\eqref{eq:var2w} and the differentiated system \eqref{eq:var1w1}--\eqref{eq:var2w1}, which will be essential to establish the exponential decay of the energy.
As a first step, we show that $\E_\delta(t)$ is equivalent to $\E_0(t)$ which is the natural energy arising in the analysis of our problem.
\begin{lemma} \label{lem:equiv}
For any $0 < \delta \le \delta^*:=\frac12\frac{\cb}{2+\CP\cb}$, we have 
\begin{align*}
\frac{1}{2} \E_0(t) \le \E_\delta(t) \le \frac{3}{2} \E_0(t) \qquad \forall t \ge 0.
\end{align*}
\end{lemma}
\begin{proof}
We start by deriving an estimate for $\|w(t)\|_\alpha$. 
Since $\d : V^1 \subset W^1 \to W^2$ is a closed linear operator, we have $N(\d)=\overline{N(\d)}$. As a consequence, we may split 
\begin{align} \label{eq:split}
w(t) = w_0(t) + w_1(t)
\end{align} 
with $w_0(t) \in N(\d)$ and $w_1(t) \in N(\d)^{\perp_\beta}$. 
From the Poincar\'e inequality~\eqref{eq:poincare}, the orthogonal splitting \eqref{eq:split}, and equation~\eqref{eq:aux2}, we immediately deduce that 
\begin{align*}
\|w_1(t)\|_{\alpha} 
\le \CP \|\d w_1(t)\|_{\gamma^{-1}} 
= \CP \|\d w(t)\|_{\gamma^{-1}} 
= \CP \|\dt w^*(t)\|_{\gamma}.
\end{align*}
In order to estimate the second component $w_0(t) \in N(\d)$ in \eqref{eq:split}, we use the orthogonality of the splitting and equation~\eqref{eq:var1w} to see that
\begin{align*}
\|w_0(t)\|_\beta^2 
&= \langle\beta w_0(t), w_0(t)\rangle_{W^1} 
  = \langle\beta w(t),w_0(t)\rangle_{W^1}\\
&= -\langle\alpha\dt w(t), w_0(t)\rangle_{W^1} + \langle w^*(t), \d w_0(t)\rangle_{W^2}\,.
\end{align*}
Since $w_0(t) \in N(\d)$, we have $\d w_0(t)=0$ and the second term drops out. 
Using the Cauchy-Schwarz inequality and the norm equivalence~\eqref{eq:norm_equivalence}, this yields
\begin{align*}
\cb \|w_0(t)\|_\alpha^2 \le \|w_0(t)\|_\beta^2 
 &\le \|\dt w(t)\|_\alpha \|w_0(t)\|_\alpha,
\end{align*}
and hence $\|w_0(t)\|_\alpha \le \frac{1}{\cb} \|\dt w(t)\|_\alpha$. 
In summary, we thus have shown that 
\begin{align} \label{eq:est}
\|w(t)\|_\alpha \le \frac{1}{\cb} \|\dt w(t)\|_\alpha + \CP \|\dt w^*(t)\|_\gamma. 
\end{align}
By elementary computations and Young's inequality, we then obtain
\begin{align*}
\delta |\langle\dt w(t), w(t) \rangle_\alpha| 
&\le \delta \|\dt w(t)\|_\alpha \|w(t)\|_\alpha \\
&\le \delta \Big(\frac{1}{\cb} + \frac{\CP}{2}\Big) \|\dt w(t)\|_\alpha^2 + \delta \frac{\CP}{2} \|\dt w^*(t)\|^2_\gamma \, . 
\end{align*}
For any $0 \le \delta \le \frac{1}{2}\frac{\cb}{2+\CP\cb}$, both leading factors can be estimated by $\tfrac{1}{4}$, and the last line can thus be bounded by $\frac{1}{2} \E_0(t)$, which already yields the assertion of the lemma.
\end{proof}

As a next step, we now show that the modified energy decays exponentially. 
\begin{lemma} \label{lem:main}
For any 
$0 \le \delta \le \delta^{**} = \min(\frac{1}{2}\frac{\cb}{2+\CP\cb}, \frac{\cb}{2+2\Cb/\cb +(\CP\Cb)^2})$, 
there holds 
\begin{align}
\ddt \E_\delta(t) \le -\frac{2\delta}{3}  \E_\delta(t), \qquad t \ge 0. 
\end{align}
\end{lemma}
\begin{proof}
From the definition of $\E_\delta(t)$ and \eqref{eq:dissw}, we get
\begin{align*}
\ddt \E_\delta(t) 
&= \ddt \frac{1}{2} \big(\|\dt w(t)\|_\alpha^2 + \|\dt w^*(t)\|_\gamma^2 \big)  + \ddt \delta \langle\dt w(t), w(t)\rangle_\alpha \\
&= -\|\dt w(t)\|_\beta^2 + \delta \|\dt w(t)\|^2_\alpha + \delta \langle\alpha\dtt w(t), w(t)\rangle_{W^1} \\
& \le -\left(\cb - \delta\right) \|\dt w(t)\|^2_\alpha + \delta \langle\alpha\dtt w(t), w(t)\rangle_{W^1}\,.
\end{align*}
Using \eqref{eq:var1w1} and \eqref{eq:var2w}, we can see that 
\begin{align*}
\langle\alpha\dtt w(t), w(t)\rangle_{W^1}    
&= -\langle\beta\dt w(t), w(t)\rangle_{W^1} + \langle \dt w^*(t), \d w(t)\rangle_{W^2} \\
&= -\langle\beta\dt w(t), w(t)\rangle_{W^1} - \|\dt w^*(t)\|^2_\gamma\,.
\end{align*}
With~\eqref{eq:est} and Young's inequality, the first term can be further bounded by
\begin{align*}
-\langle\beta\dt w(t), w(t)\rangle_{W^1}
&\le C_\beta\|\dt w(t)\|_\alpha \|w(t)\|_\alpha \\
&\le \Big(\frac{C_\beta}{\cb} + \frac{(\CP \Cb)^2}{2}\Big) \|\dt w(t)\|_\alpha^2 + \frac{1}{2} \|\dt w^*(t)\|_\gamma^2\, .
\end{align*}
Now let $C_1$ denote the coefficient in front of the first term. Then together with the previous estimates, we immediately obtain
\begin{align*}
\ddt \E_\delta(t) 
\le -\big(\cb - \delta(1+C_1)\big) \|\dt w(t)\|^2_\alpha - \frac{\delta}{2}  \|\dt w^*(t)\|^2_\gamma\,.
\end{align*}
For any 
$0 \le \delta \le\frac{\cb}{2+2\Cb/\cb +(\CP\Cb)^2}$, we further see that $\delta (1+C_1) \le \frac{\cb}{2}$, and consequently 
\begin{align} \label{eq:exp}
\ddt \E_\delta(t)
&\le -\min(\cb,\delta) \E_0(t) 
\le - \min(\cb,\delta) \frac{2}{3} \E_\delta(t)\,.
\end{align}
In the last step, we used the right estimate of Lemma~\ref{lem:equiv}, and thus the condition $\delta \le \delta^*$, which particularly implies $\delta \le\cb$.
This yields the assertion of the lemma.
\end{proof}

\subsection*{Proof of Theorem~\ref{thm:main}}

From ~\eqref{eq:exp}, Grönwall's inequality \cite{Wloka}, and Lemma~\ref{lem:equiv}, we obtain for any $0\le s\le t$ the estimate
\begin{align}
\frac{1}{2} \E_0(t) \le \E_\delta(t) \le e^{-\frac{2\delta}{3} (t-s)} \E_\delta(s) \le \frac{3}{2} e^{-\frac{2\delta}{3} (t-s)} \E_0(s)\,,
\end{align}
 From the definition of $\E_0(t)$ and $(w,w^*)$, we see that $\E_0(t) = \frac{1}{2} (\|u(t)\|_\alpha^2 + \|u^*(t)\|_\gamma^2)$. We then choose $\delta$ as large as possible and obtain \eqref{eq:decay} with $C'=3$ and $c'=\frac{2}{3}\delta^{**}$, which was defined in the previous lemma. 
This proves the assertion of Theorem~\ref{thm:main} for classical solutions. 
By Remark~\ref{rem:dense}, the estimate remains valid for mild solutions as well.
\qed

\section{Compatible discretization}
\label{sec:disc}

We will show in the following that exponential stability can be preserved for numerical approximations obtained by appropriate discretization strategies. Most of the arguments used on the continuous level carry over verbatim, and we therefore only sketch the main additional assumptions and differences required for the analysis.

\subsection{Discretization in space}
We utilize a conforming Galerkin approximation of the weak form \eqref{eq:var1}--\eqref{eq:var2} of our problem. 
Let $V_h^1 \subset V^1$ and $W_h^2 \subset W^2$ be finite dimensional and  consider discrete solutions $(u_h,u_h^*) \in C^1([0,\infty);V_h^1 \times W_h^2)$ of 
\begin{alignat}{2}
\langle\alpha\dt u_h(t), v_h\rangle_{W^1} - \langle u_h^*(t), \d v_h\rangle_{W^2} &= - \langle\beta u_h(t), v_h\rangle_{W^1}, \label{eq:var1h} \\
\langle\gamma \dt u_h^*(t), v_h^*\rangle_{W^2} + \langle \d u_h(t), v_h^*\rangle_{W^2} &= 0, \label{eq:var2h}
\end{alignat}
for all test functions $v_h \in V_h^1$, $v^*_h \in W_h^2$, and all $t \ge 0$.
From the properties of $\alpha$ and $\gamma$, this can be seen to make up a regular system of linear ordinary differential equations. Well-posedness can thus be deduced from the Picard-Lindelöf theorem. 
\begin{lemma} \label{lem:41}
For any choice of initial values $u_h(0) \in V_h^1$ and $u_h^*(0) \in W_h^2$, the linear system \eqref{eq:var1h}--\eqref{eq:var2h} has a unique solution $(u_h,u_h^*)\in C^1([0,\infty);V_h^1 \times W_h^2)$. Moreover
\begin{align} \label{eq:dissh}
\ddt \frac{1}{2} \Big( \|u_h(t)\|_\alpha^2 &+ \|u_h^*(t)\|_\gamma^2 \Big) 
= -\|u_h(t)\|_\beta^2 \le 0, \qquad t \ge 0.
\end{align}
\end{lemma}
The energy estimate follows by testing \eqref{eq:var1h}--\eqref{eq:var2h} with $v_h=u_h(t)$ and $v_h^*=u_h^*(t)$, and the same arguments as used on the continuous level.
We continue by defining 
\begin{align} \label{eq:auxh}
w_h(t) = w_h(0) + \int_0^t u_h(s) \, ds 
\quad \text{and} \quad 
w_h^*(t) = w_h^*(0) + \int_0^t u_h^*(s) \, ds,
\end{align}
with initial values $w_h(0)$, $w_h^*(0)$ making up a solution of 
\begin{alignat}{2} 
\langle \beta w_h(0), v_h\rangle_{W^1} - \langle w_h^*(0), \d v_h\rangle_{W^2} &= -\langle \alpha u_h(0), v_h\rangle_{W^1} \qquad && \forall v_h \in V_h^1 \label{eq:aux1h}\\
\langle \d w_h(0), v_h^*\rangle_{W^2} &= - \langle \gamma u_h^*(0), v_h^*\rangle_{W^2} \qquad &&\forall v_h^* \in W_h^2. \label{eq:aux2h}
\end{alignat}
To ensure existence of a solution $(w_h(0),w_h^*(0)) \in V_h^1 \times W_h^2$, we require a compatibility condition on the initial value $u_h^*(0)$ and an additional compatibility condition 
\begin{align} \label{eq:compat}
\d V_h^1 \subset W_h^2
\end{align}
on the discretization spaces. This allows us to split 
$W_h^2 = \d V_h^1 \oplus Y_h^2$,
with the orthogonal complement $Y_h^2 = \{v_h^* \in W_h^2 : \langle \gamma v_h^*,\d v_h\rangle_{W^2}=0 \ \forall v_h \in V_h^1\}$.
With very similar arguments as used on the continuous level, we obtain the following result. 
\begin{lemma} \label{lem:42}
Let $\d V_h^1 \subset W_h^2$ and assume that \begin{align} \label{eq:compath}
\langle \gamma u_h^*(0), y_h^* \rangle_{W^2}=0 \qquad  \text{for all } y_h^* \in Y_h^2.
\end{align}
Then the system \eqref{eq:aux1h}--\eqref{eq:aux2h} admits a 
unique solution $(w_h(0),w_h^*(0)) \in V_h^1 \times Y_h^2$. 
Moreover, the function $(w_h,w_h^*)$ defined in \eqref{eq:auxh} satisfies 
\begin{alignat}{2}
\langle\alpha\dt w_h(t), v_h\rangle_{W^1} - \langle w_h^*(t), \d v_h\rangle_{W^2} &= - \langle\beta w_h(t), v_h\rangle_{W^1}, \label{eq:var1hw} \\
\langle\gamma \dt w_h^*(t), v_h^*\rangle_{W^2} + \langle \d w_h(t), v_h^*\rangle_{W^2} &= 0, \label{eq:var2hw}
\end{alignat}
for all $v_h \in V_h^1$, $v_h^* \in W_h^2$, and $t \ge 0$, as well as 
\begin{alignat}{2}
\langle\alpha\dtt w_h(t), v_h\rangle_{W^1} - \langle \dt w_h^*(t), \d v_h\rangle_{W^2} &= - \langle\beta \dt w_h(t), v_h\rangle_{W^1}, \label{eq:var1hww} \\
\langle\gamma \dtt w_h^*(t), v_h^*\rangle_{W^2} + \langle \d \dt w_h(t), v_h^*\rangle_{W^2} &= 0. \label{eq:var2hww}
\end{alignat}
\end{lemma}
For the discrete stability analysis, we require a discrete Poincar\'e inequality 
\begin{align} \label{eq:poincareh}
\|v_h\|_\alpha \le C_{\mathrm{P},h} \|\d v_h\|_{\gamma^{-1}} \qquad \forall v_h \in Z_h^{\perp_\beta},
\end{align}
where $Z_h=\{z_h \in V_h^1 : \d z_h = 0\}$ and $Z_h^{\perp_\beta} = \{v_h \in V_h^1 : \langle \beta v_h,z_h\rangle_{W^1} = 0 \ \forall z_h \in Z_h\}$. 
Validity of this condition can be established by the use of a bounded cochain projection; see \cite[Ch.~7]{Arnold2018} for details. 
We can now follow the proof of Theorem~\ref{thm:main} step-by-step to obtain the following discrete stability result.
\begin{theorem}
Let $\d V_h^1 \subset W_h^2$ and $(u_h,u_h^*)$ denote any solution of \eqref{eq:var1h}--\eqref{eq:var2h} satisfying the compatibility condition \eqref{eq:compath}. Then 
\begin{align}
\|u_h(t)\|_\alpha^2 + \|u_h^*(t)\|_{\gamma}^2 \le C'' e^{-c''(t-s)} \big( \|u_h(s)\|_\alpha^2 + \|u_h^*(s)\|_{\gamma}^2\big) \qquad \forall 0 \le s \le t,
\end{align}
with constants $C''$, $c''>0$ depending on the discrete Poincar\'e constant $C_{\mathrm{P},h}$, but otherwise independent of the spaces $V_h^1$, $W_h^2$, and the particular solution $(u_h,u_h^*)$.
\end{theorem}

\subsection{Time discretization}

As a second result of this section, we show that exponential stability can also be preserved under appropriate discretization in time. 
Let $\tau>0$ be a fixed time step and set $t_n = n \tau$ for $n \ge 0$. We denote by $u_n \approx u(t_n)$ approximations of a function $u$ at the discrete time points, and write 
$$
\dtau u_n = \frac{1}{\tau} (u_n - u_{n-1}) 
$$
for the backwards difference quotient. 
For the approximation of \eqref{eq:1}--\eqref{eq:2}, we then consider sequences $(u_n,u^*_n)$, $n \ge 1$, defined recursively by 
\begin{alignat}{2}
\alpha \dtau u_n - \d^* u_n^* &= -\beta u_n, \qquad && \text{in } W^1, \ n \ge 1, \label{eq:1tau}\\
\gamma \dtau u_n^* + \d u_n &= 0 \qquad && \text{in } W^2, \ n \ge 1. \label{eq:2tau}
\end{alignat}
Appropriate initial values $u_0$ and $u^*_0$ have to be provided. 
Under the general assumptions stated in Section~\ref{sec:prelim}, we obtain the following result.
\begin{lemma}
For any given $(u_0,u_0^*) \in W^1 \times W^2$, the problem \eqref{eq:1tau}--\eqref{eq:2tau} defines a unique sequence $(u_n,u^*_n) \in V^1 \times V_2^*$, $n \ge 1$. 
Moreover, 
\begin{align*}
\dtau\frac12 \big( \|u_n\|_\alpha^2 + \|u_n^*\|_\gamma^2) \le -\|u_n\|_\beta^2 \le 0. 
\end{align*}
\end{lemma}
\begin{proof}
Existence of a unique solution for every time step follows from another generalization of Brezzi's lemma; see \cite[Thm.~4.3.1]{Boffi}. 
For the energy estimate, we note that $\dtau \|u_n\|_\alpha^2 = 2 \langle \alpha \dtau u_n,u_n\rangle_{W^1} - \tau \|\dtau u_n\|_\alpha^2 \le 2\langle \alpha \dtau u_n,u_n\rangle_{W^1}$, and similarly $\dtau \|u_n^*\|_\gamma^2 \le 2\langle \gamma \dtau u_n^*, u_n^*\rangle_{W^2}$. The estimate then follows immediately by testing \eqref{eq:1tau}--\eqref{eq:2tau} with $v=u_n$ and $v^* = u_n^*$, using the same arguments as on the continuous level.  
\end{proof}
For the proof of exponential stability, we again introduce discrete primitives by 
\begin{align}\label{eq:auxn}
w_n = w_0 + \tau\sum\nolimits_{k=1}^n u_k  
\qquad \text{and} \qquad 
w_n^* = w^*_0 + \tau\sum\nolimits_{k=1}^n u^*_k. 
\end{align}
The initial values $w_0$ and $w_0^*$ are chosen like in Section~\ref{sec:proof}.
From the linearity of the problem and the use of equidistant time steps, we readily obtain the following result. 
\begin{lemma}
The sequence $(w_n,w_n^*)$ defined in \eqref{eq:auxn} satisfies 
\begin{alignat}{2}
\alpha \dtau w_n - \d^* w_n^* &= -\beta w_n\,, \qquad && n \ge 1, \label{eq:1tauw}\\
\gamma \dtau w_n^* + \d w_n &= 0\,, \qquad && n \ge 1. \label{eq:2tauw}
\end{alignat}
Furthermore, there holds
\begin{alignat}{2}
\alpha \ddtau w_n - \d^* \dtau w_n^* &= -\beta \dtau w_n\,, \qquad && n \ge 2, \label{eq:1tauww}\\
\gamma \ddtau w_n^* + \d \dtau w_n &= 0\,, \qquad && n \ge 2, \label{eq:2tauww}
\end{alignat}
where $\ddtau u_n = \frac{1}{\tau^2} (u_n - 2 u_{n-1} + u_{n-2})$ is the second backward difference quotient.
\end{lemma}
One can then again walk through the proof of Theorem~\ref{thm:main} step-by-step, which leads to the following stability results for the time-discrete problem. 
\begin{theorem}
Under the assumptions of Theorem~\ref{thm:main}, the discrete solutions $(u_n,u_n^*)$ obtained by \eqref{eq:1tau}--\eqref{eq:2tau} with initial value $(u_0,u_0^*)=(u(0),u^*(0))$ satisfy 
\begin{align*}
\bigl(\|u_n\|_\alpha^2 + \|u_n^*\|_\gamma^2\bigr) \le C' e^{-c' (t_n - t_m)} \big( \|u_m\|_\alpha^2 + \|u_m^*\|_\gamma^2\big) \qquad \forall 0 \le m \le n. 
\end{align*}
The constants $C'$, $c'>0$ can be chosen the same as in the proof of Theorem~\ref{thm:main}.
\end{theorem}

\subsection*{Concluding remarks}

It is possible to combine a compatible Galerkin approximation with an implicit time discretization. The resulting fully discrete scheme still retains the stability properties of the continuous problem. 
One can also formulate higher order time discretization schemes which preserve the exponential stability. 
The stability of the discrete problem further allows to derive discretization error estimates which are uniform in the time variable. 
We refer to \cite{Kugler2018} for some results in these directions.

\section{Examples}
\label{sec:examples}

To illustrate the wide applicability of our main results, we now discuss a few typical examples that fit into our abstract framework and discuss the assumptions needed for our analysis and the compatible approximation.

\subsection{Electrodynamics}

Let us return to Maxwell's equations
\begin{alignat}{2}
\eps \dt E - \curl H &= -\sigma E  \qquad &&\text{in } \Omega, \ t \ge 0, \label{eq:max1}\\
\mu \dt H + \curl_0 E &= 0 \qquad  && \text{in } \Omega, \ t \ge 0  \label{eq:max2}
\end{alignat}
already discussed in the introduction.
The subscript in $\curl_0 E$ means that the $E \times n=0$ is required on the boundary. 
This example fits into our abstract framework with spaces and operators defined by $W^1=W^2=L^2(\Omega)^3$, $\d=\curl_0$, $V^1=H_0(\curl)$, $\d^*=\curl$, $V_2^* = H(\curl)$. 
The solution components $u=E$ and $u^*=H$ correspond to the electric and magnetic field intensities. 
Furthermore,  $\alpha u =\eps E$, $\beta u =\sigma E$, $\gamma u^*=\mu H$ are defined by multiplication with the corresponding material parameters.  
The assumptions on the operators $\alpha$, $\beta$, $\gamma$ are met, e.g., if $\eps$, $\sigma$, $\mu$ are uniformly positive and bounded. 

As already mentioned in Section~\ref{sec:intro}, the compatibility condition $\gamma u^*(0) \in \mathcal{R}(\d)$ here means that $B(0)=\mu H(0) = \curl A(0)$ for some vector potential $A(0) \in H_0(\curl)$. 
This implies that $\div B(0)=0$ in $\Omega$ and $B(0) \cdot n = 0$ on $\partial\Omega$. Due to equation~\eqref{eq:max2}, these conditions remain valid for all time $t \ge 0$. Under this natural condition, exponential decay of the solution follows immediately from our abstract theory. 

To obtain a compatible space discretization, we consider a finite-element approximation on a conforming tetrahedral mesh $\T_h$ of $\Omega$. We choose $V_h^1=\N_k(\T_h) \cap H_0(\curl)$ and $W_h^2 = P_k(\T_h)^3$, i.e., by N\'ed\'elec and discontinuous finite elements of order $k$, respectively; see \cite{Boffi,Nedelec80} fro details. 
For this choice, the condition $\d V_h^1 \subset W_h^2$ is valid. Furthermore, a discrete Poincar\'e inequality \eqref{eq:poincareh} holds with $C_{\mathrm{P},h}$ depending only on the shape regularity of the mesh $\T_h$ and on the polynomial degree; see \cite{Arnold2018,Monk}. 
The numerical approximations obtained with this strategy again decay exponentially with a similar rate as the continuous solution.
Using the close connection of mixed finite element methods with FIT or FDTD methods \cite{Cohen02,Radu22}, similar decay results could be obtained for these kind of methods, at least for the semi-discretization in space.
\begin{remark}
In passing by we remark that our abstract framework also captures certain linear Kirchhoff networks. For instance, consider a network whose branches consist of resistances and inductances, while the nodes are connected by capacitances to a ground node. Then $u$ is the vector of branch currents, $u^*$ is the vector of nodal potentials, $\d$ is the discrete divergence, $\d^*$ is the discrete gradient, $\alpha$ is the symmetric positive definite inductance matrix, which entails self and mutual branch inductances, $\beta$ the positive diagonal resistance matrix, and $\gamma$ the positive diagonal node capacitance matrix. Equations \eqref{eq:1} and \eqref{eq:2} are Kirchhoff's voltage and current laws, respectively.
\end{remark}

\subsection{Vibration of a membrane}
As a different area of application, let us consider the vibration of a membrane which is assumed fixed across a flat frame. 
The vertical deflection of the membrane can be described by the system
\begin{alignat}{2}
\rho_0 \dt v - \div \sigma &= -c v && \qquad \text{in } \Omega, \ t \ge 0,\\
\kappa^{-1} \dt \sigma - \nabla_0 v &= 0 && \qquad \text{in } \Omega, \ t \ge 0.
\end{alignat}
Here $v$ is the vertical velocity and  $\sigma$ represents the tension forces inside the membrane. 
Furthermore, $\rho_0$ is related to the inertia, $\kappa$ is the stiffness of the membrane, and $c$ denotes a friction coefficient, which represents the damping through the surrounding medium. 
This problem again fits into our abstract setting: Here $W^1=L^2(\Omega)$, $W^2=L^2(\Omega)^2$, and $\d=-\nabla_0$ is the gradient with zero boundary conditions; its domain is $V^1=H_0^1(\Omega)$. The adjoint operator is $\d^* = \div$ with domain $V_2^*=H(\div)$; see \cite{Arnold2018}. 
The material operators $\alpha$, $\beta$, $\gamma$ again amount to multiplication with the corresponding coefficients.

The compatibility conditions for the initial value here reads $\kappa^{-1} \sigma(0) = \nabla_0 u(0)$ for some $u(0) \in H_0^1(\Omega)$. This condition is met, for instance, for initial value $\sigma(0)=0$. 
By our main theorem, we then obtain exponential decay of the solution to steady state.

A compatible Galerkin approximation can be obtained as follows: Let $\T_h$ denote a conforming triangulation of the domain $\Omega$. We choose $V_h^1 = \P_{k+1}(\T_h) \cap H_0^1(\Omega)$ and $W_h^2 = \P_k(\T_h)^2$ consisting of continuous resp.~discontinuous finite elements with appropriate polynomial degree. These spaces meet the compatibility condition $\d V_h^1 \subset W_h^2$, and we can again predict the exponential decay of the discrete solutions.

\subsection{Elastodynamics}

We consider the propagation of waves in a viscoelastic medium of Maxwell-type with deformation fixed at the boundary. 
This can be modeled by 
\begin{alignat}{2}
A \dt \sigma - \eps_0(v) &= -B\sigma\qquad && \text{in } \Omega, \ t\ge 0,\\
\rho_0 \dt v - \div \sigma &= 0 \qquad\qquad &&\text{in } \Omega, \ t \ge 0. 
\end{alignat}
Here $A$ is the inverse of Hooke's tensor, $B$ represents the friction law, and $\rho_0$ denotes the mass density of the body~\cite{Auld1990,Wriggers2008}. 
Furthermore, $v$ denotes the displacement velocity, $\sigma$ the Cauchy tress tensor, and $\eps_0(v)$ represents the linearized strain rate tensor complemented by homogeneous boundary conditions. 
This problem again fits into our abstract setting: Here $W^1=L^2(\Omega,\mathbb{R}^{3 \times 3}_\mathrm{sym})$ is the space of square integrable symmetric tensor fields, $W^2=L^2(\Omega)^3$, and $\d = \div$ is the row-wise divergence with domain $V^1 = H(\div;\RR^{3 \times 3}_\mathrm{sym})$. The adjoint operator is $\d^*=\eps_0(\cdot)$ with domain $V_2^*=H_0^1(\Omega)^3$. The spaces and operators thus form a segment of the elasticity complex \cite{Arnold2010,Pauly2023}.  

The compatibility condition for the initial value here reads $\rho_0 v(0) = \div \sigma'(0)$ with some $\sigma'(0) \in V^1$.
Validity of this conditions follows immediately from the exactness of the elasticity complex \cite{Arnold2010,Pauly2023}.  By application of our abstract results, we thus obtain exponential convergence of the system to steady state.

A compatible Galerkin approximation by mixed finite elements can here be obtained as follows. Let $\T_h$ be a conforming tetrahedral mesh. We then choose piecewise polynomials $W_h^2=P_k(\T_h)^3$ of order $k$ for the velocities and $V_h^1 = \AAW_k(\T_h) \cap H(\div;\RR^{3 \times 3}_\mathrm{sym})$ for the stresses; here $\AAW_k$ denotes the stress element of order $k$ developed by Arnold, Awanou, and Winther~\cite{Arnold2008}.
This choice of elements satisfies $\d V_h^1 \subset W_h^2$ and a discrete Poincar\'e inequality \eqref{eq:poincareh} holds with $C_{\mathrm{P},h}$ independent of the mesh size. We can thus again predict the uniform exponential decay for the discrete approximations.

\subsection{Acoustics}

We consider the propagation of sound waves in a closed cavity between two parallel plates. 
If the distance of the two plates is very small compared to their elongation, we may use a two-dimensional model, which reads 
\begin{alignat}{2}
\rho_0^{-1} \dt m + \nabla \hat p &= -c_f m \qquad && \text{in } \Omega, \ t \ge 0,\\
c_0^{-2} \dt \hat p + \div_0 m &= 0 \qquad && \text{in } \Omega, \ t \ge 0. 
\end{alignat}
Here $\rho_0$ is the density of the background medium, $c_0$ the speed of sound, $m = \rho_0 v$ the momentum density, and  $\hat p = p/\rho_0$ the kinematic pressure. 
The coefficient $c_f$ describes damping due to friction at the surface of the enclosing plates. 
This problem again fits into the abstract setting considered in this paper: Here 
$W^1=L^2(\Omega)^2$, $W^2=L^2(\Omega)$, and $\d = \div_0$ with zero boundary conditions $m \cdot n=0$ on $\partial\Omega$ and domain $V^1=H_0(\div)$. The corresponding adjoint operator is $\d^* =-\nabla$ with domain $V_2^*=H^1(\Omega)$; see \cite{Arnold2018}. 
The material operators $\alpha$, $\beta$, $\gamma$ again amount to multiplication with the parameters.

The compatibility condition for the initial values here reads $c_0^{-2} p(0)=\div M(0)$ for some $M(0) \in H_0(\div)$. This implies that 
$\int_\Omega c_0^{-2} p(0) = \int_\Omega \div M(0) = \int_{\partial\Omega} M(0) \cdot n = 0$.
The pressure average thus is the persistent mode which, however, can be fixed to zero without loss of generality.
By our abstract theory, we can then guarantee exponential convergence of solutions to steady state.

To obtain a compatible finite element approximation, let $\T_h$ be a conforming triangulation of $\Omega$. Then choose $V_h^1=\RT_k(\T_h) \cap H_0(\div)$ and $W_h^2=\P_k(\T_h)$, which are the Raviart-Thomas and discontinuous finite elements of order $k$. 
For this choice, the condition $\d V_h^1 \subset W_h^2$ is again valid; see \cite{Boffi} for details. 
As a consequence, we obtain exponential decay also for the finite-element approximations of our problem.

\begin{remark}
A similar one-dimensional system can be used to model the propagation of acoustic waves in pipes. 
Together with appropriate coupling conditions, one can then describe sound propagation in pipe networks. 
Corresponding results concerning exponential decay of the continuous and discrete solutions were derived in \cite{Kugler2018}. 
\end{remark}

\section{Discussion}

In this paper, we proved exponential stability of a class of wave propagation problems with uniform damping. 
As demonstrated by some examples, the analysis was done in the framework of Hilbert complexes, which offers all the structure and ingredients needed in our proofs, i.e., a weak characterization of solutions, energy estimates, and a generalized Poincar\'e inequality.  
The few and elementary  assumptions can be verified for a variety of applications as well as compatible discretizations thereof. 

Before we close the presentation, let us briefly mention some possible extensions, which might be worth further consideration: 
To simplify the implementation, the application of non-conforming approximations may be desireable \cite{Arnold2014,Arnold2014b}. %
While higher order approximations in space are straight forward, the extension of our results to higher order time-stepping schemes may require a more delicate analysis; see \cite{EggerKugler2018,Kugler19} for some results in this direction. 
In view of \cite{Egger2020,Kugler19}, also the extension to nonlinear damping terms seems possible. The consideration of nonlinear energies seems more difficult but would have interesting applications, e.g., in high intensity ultrasound or nonlinear optics.
An additional topic of interest is the uniformity of the decay rate in quasi-static limits. These are of relevance in electrodynamics as well as in acoustics and elastodynamics. Again, some preliminary results in this direction can be found in  \cite{EggerKugler2018,Kugler19}.

{\small 
\section*{Acknowledgements}

This work was partially supported by the international Collaborative Research Center CREATOR jointly funded by FWF and DFG. 
An essential step in clarifying the theoretical foundations could be made due to the inspiring atmosphere  at the Oberwolfach  Workshop on "Hilbert Complexes: Analysis, Applications, and Discretizations" 2022 organized by Ana Alonso, 
Doug Arnold, Dirk Pauly, and 
Francesca Rapetti. 
}


\end{document}